\documentclass{article}
\usepackage{amsmath,amssymb,amsthm}
\usepackage{mathrsfs}
\usepackage[colorlinks=true]{hyperref}

\def\R{\mathbb{R}}
\def\N{\mathbb{N}}
\def\Z{\mathbb{Z}}
\def\C{\mathbb{C}}

\newcommand{\Op}{\mathrm{Op}}
\newcommand{\OpW}{\mathrm{Op^{\mathcal{W}}}}

\renewcommand{\leq}{\leqslant}

\newtheorem{theorem}{Theorem}

\newtheorem{proposition}{Proposition}
\newtheorem{corollary}{Corollary}
\newtheorem{definition}{Definition}
\newtheorem{lemma}{Lemma}
\theoremstyle{definition}\newtheorem{example}{Example}
\theoremstyle{definition}\newtheorem{remark}{Remark}

\title{Quantum ergodicity and quantum limits for sub-Riemannian Laplacians}

\author{Yves Colin de Verdi\`ere\footnote{Universit\'e de Grenoble,
Institut Fourier, Unit{\'e} mixte de recherche CNRS-UJF 5582, BP 74, 38402-Saint Martin d'H\`eres Cedex, France (\texttt{yves.colin-de-verdiere@ujf-grenoble.fr}).}
\and
Luc Hillairet \footnote{Universit\'e d'Orl\'eans, F\'ed\'eration Denis Poisson, Laboratoire MAPMO, route de Chartres, 45067 Orl\'eans Cedex 2, France (\texttt{luc.hillairet@univ-orleans.fr}).}
\and
Emmanuel Tr\'elat\footnote{Sorbonne Universit\'es, UPMC Univ Paris 06, CNRS UMR 7598, Laboratoire Jacques-Louis Lions, Institut Universitaire de France, F-75005, Paris, France (\texttt{emmanuel.trelat@upmc.fr}).}
}

\begin{document}

\maketitle

\begin{abstract}
This paper is a proceedings version of \cite{CHT-I}, in which we state a Quantum Ergodicity (QE) theorem on a 3D contact manifold, and in which we establish some properties of the Quantum Limits (QL).

We consider a sub-Riemannian (sR) metric on a compact 3D manifold with an oriented contact distribution. There exists a privileged choice of the contact form, with an associated Reeb vector field and a canonical volume form that coincides with the Popp measure. We state a QE theorem for the eigenfunctions of any associated sR Laplacian, under the assumption that the Reeb flow is ergodic. The limit measure is given by the normalized canonical contact measure.
To our knowledge, this is the first extension of the usual Schnirelman theorem to a hypoelliptic operator.
We provide as well a decomposition result of QL's, which is valid without any ergodicity assumption.
We explain the main steps of the proof, and we discuss possible extensions to other sR geometries.
\end{abstract}

\section{Introduction}
The property of Quantum Ergodicity (QE) is defined as follows.
Let $M$ be a metric space, endowed with a measure $\mu$ defined on a suitable compactification of $M$, and let $T$ be a self-adjoint operator on $L^2(M,\mu)$, bounded below and having a compact resolvent (and hence a discrete spectrum). Let $(\phi_n)_{n\in\N^*}$ be a (complex-valued) Hilbert basis of $L^2(M,\mu)$, consisting of eigenfunctions of $T$, associated with the ordered sequence of eigenvalues $\lambda_1\leq\cdots\leq\lambda_n\leq\cdots$. 
We say that QE holds for $T$ is there exist a probability measure $\nu$ on $M$ and a density-one sequence $(n_j)_{j\in\N^*}$ of positive integers such that the sequence of probability measures $|\phi_{n_j}|^2 d\mu$ converges weakly to $\nu$.
The measure $\nu$ may be different from some scalar multiple of $\mu$, and may even be singular with respect to $\mu$.

Microlocal versions of QE hold true in general and are stated in terms of pseudo-differential operators, yielding a result of the kind
$$
\langle\Op(a)\phi_{n_j},\phi_{n_j}\rangle \rightarrow \int_\Sigma a\, d\tilde\nu ,
$$
for every classical symbol $a$ of order $0$, where $\tilde\nu$ is a lift to $T^\star M$ of the measure $\nu$ and $\Sigma$ is a subset of $T^\star M$ (that have to be identified).

Such a QE property provides an insight on the way eigenfunctions concentrate when considering highfrequencies.

\medskip

In the existing literature, QE results exist in different contexts, but always for elliptic operators. The first historical QE result is due to Shnirelman in 1974 (see \cite{Shn-74} for a sketch of proof), stating that, on a compact Riemannian manifold $(M,g)$, if the geodesic flow is ergodic, then we have QE for any orthonormal basis of eigenfunctions of the Laplace-Beltrami operator, with $d\nu$ equal to the normalized Riemannian volume (and $d\tilde\nu$ equal to the normalized Liouville measure).
The complete proof has been obtained later in \cite{yCdV-85,Zel-87}.
Then, the Shnirelman theorem has been extended to the case of manifolds with boundary in \cite{GL-93,ZZ-96}, to the case of discontinuous metrics in \cite{JSSC-14} and to the semiclassical regime in \cite{HMR-87}.

\medskip

We provide here a Shnirelman theorem in a hypoelliptic context. 

\begin{theorem}\label{thm1}
Let $M$ be compact connected 3D manifold (without boundary), endowed with an arbitrary volume form $d\mu$ and with a Riemannian contact structure. We consider the associated sub-Riemannian Laplacian $\triangle_{sR}$.
If the Reeb flow is ergodic for the Popp measure, then we have QE, and the limit measure is the probability Popp measure.
\end{theorem}

Several definitions are due to understand the content of this theorem. In Section \ref{sec2}, we are thus going to recall:
\begin{itemize}
\item what is a sub-Riemannian Laplacian;
\item what are a contact structure, the Reeb flow, and the Popp measure.
\end{itemize}
We will then provide a more general (microlocal) version of Theorem \ref{thm1}.
We also provide a second result, valid without any ergodicity assumption, stating that every quantum limit that is supported on the characteristic manifold of the sR Laplacian is invariant under the Reeb flow, and that most QL satisfy the previous assumption on their support.

We will explain the main steps of the proof in Section \ref{sec3}, and we will conclude in Section \ref{sec4} by giving some perspectives and open problems.

\section{Sub-Riemannian Laplacians in the 3D contact case}\label{sec2}
\subsection{General definition}
We first recall that a sub-Riemannian (sR) structure is a triple $(M,D,g)$, where $M$ is a smooth manifold, $D$ is a smooth subbundle of $TM$ (called horizontal distribution), and $g$ is a fibered Riemannian metric on $D$.

\begin{example}
We speak of a contact structure when $M$ is of dimension $3$ and $D$ is a contact distribution, that is, we can write $D=\ker\alpha$ locally around any point, for some one-form $\alpha$ such that $\alpha\wedge d\alpha\neq 0$ (locally). The local Darboux model in $\R^3$ at the origin is given by $\alpha = dz+x\, dy$, and then, locally, we can write $D=\mathrm{Span}(X,Y)$, with $X$ and $Y$ the vector fields defined by $X=\partial_x$ and $Y=\partial_y-x\partial_z$. Note that $D$ is of codimension one, and the Lie bracket $[X,Y]=-\partial_z$ generates the missing direction.
\end{example}

\paragraph{Sub-Riemannian Laplacian.}
In order to define a sub-Riemannian Laplacian $\triangle_{sR}$, let us choose a smooth volume form $d\mu$ on $M$, the associated measure being denoted by $\mu$.
Let $L^2(M,\mu)$ be the set of complex-valued functions $u$ such that $|u|^2$ is $\mu$-integrable over $M$.
Then, in whole generality, $-\triangle_{sR} $ is the nonnegative self-adjoint operator on $L^2(M,\mu)$ defined as the Friedrichs extension of the Dirichlet integral
$$
Q(\phi)=\int _M \Vert d\phi \Vert_{g^*}^2 \, d\mu, 
$$
where the norm of $d\phi $ is calculated with respect to the co-metric $g^\star$ on $T^\star M $ associated with $g$.
The sR Laplacian $\triangle_{sR}$ depends on the choice of $g$ and of $d\mu$.

Let us give two equilavent definitions of $\triangle_{sR}$:
\begin{itemize}
\item We have
$$
\boxed{\triangle_{sR}\phi = \mathrm{div}_\mu(\nabla_{sR}\phi)}
$$
for every smooth function $\phi$ on $M$,
where $\mathrm{div}_\mu$ is the divergence operator associated with the volume form $d\mu$, defined by $\mathcal{L}_X d\mu = \mathrm{div}_\mu(X)\, d\mu$ for any vector field $X$ on $M$, and where the horizontal gradient $\nabla_{sR}\phi$ of a smooth function $\phi$ is the unique section of $D$ such that $g_q(\nabla_{sR}\phi(q),v)=d\phi(q).v$, for every $v\in D_q$.

\item If $(X_1,\ldots,X_m)$ is a local $g$-orthonormal frame of $D$, then $\nabla_{sR}\phi=\sum_{i=1}^m (X_i\phi)X_i$, and
$$
\boxed{\triangle_{sR}= - \sum_{i=1}^m X_i^\star X_i = \sum_{i=1}^m \left( X_i^2 + \mathrm{div}_\mu(X_i) X_i \right) } 
$$

\end{itemize}

\paragraph{Hypoellipticity.}
A well known theorem due to H\"ormander (see \cite{Ho-67}) states that, under the assumption $\mathrm{Lie}(D) = TM$, the operator $\triangle_{sR}$ is hypoelliptic (and even subelliptic), has a compact resolvent and thus a discrete spectrum $\lambda_1\leq\lambda_2\leq\cdots\leq\lambda_n\rightarrow+\infty$. 

\begin{remark}
Denoting by $h_X(q,p)=\langle p,X(q)\rangle$ the Hamiltonian associated with a vector field $X$ (in canonical coordinates), the principal symbol of $-\triangle_{sR}$ is 
$$
\sigma_P(-\triangle_{sR}) = \sum_{i=1}^m h_{X_i}^2 = g^\star ,
$$ 
(it coincides with the co-metric $g^\star$), and its sub-principal symbol is zero. 

Actually, in our main result, $ \triangle_{sR}$ may be any self-adjoint second-order differential operator whose principal symbol is $g^\star$, whose sub-principal symbol vanishes.
\end{remark}

\begin{remark}
The Hamiltonian function $\frac{1}{2}g^\star$ induces a Hamiltonian vector field $X_g$ which generates the so-called normal geodesics of the sR structure (see \cite{Mo-02}).
\end{remark}

\subsection{The 3D contact case}
We assume that the manifold $M$ is compact and of dimension $3$, and that $D$ is an oriented contact distribution, that is, we can write $D=\ker\alpha$ for some one-form $\alpha$ globally defined such that $\alpha\wedge d\alpha\neq 0$ (see \cite{CHT-I} for the extension to the non-orientable case).

In order to define the canonical contact measure and the Reeb vector field, we need to normalize the contact form defining $D$.
There exists a unique contact form $\alpha_g$ such that $d\alpha_g (X,Y)=1$ for any positive $g$-orthonormal local frame $(X,Y)$ of $D$. Equivalently, ${(d\alpha_g)}_{\vert D}$ coincides with the volume form induced by $g$ on $D$.

\paragraph{Popp measure.}
We define the density $dP= \vert\alpha_g \wedge d\alpha_g\vert$ on $M$. In general, $dP$ differs from $d\mu$. The corresponding measure $P$ is called the \textit{Popp measure} in the existing literature (see \cite{Mo-02}, where it is defined in the general equiregular case). Of course, here, this measure is the canonical contact measure associated with the normalized contact form $\alpha_g$.
We also define the \textit{probability Popp measure}
\begin{equation*}
\nu = \frac{P}{P(M)} .
\end{equation*}

\paragraph{Reeb flow.}
The \textit{Reeb vector field} $Z$ of the contact form $\alpha_g$ is defined as the unique vector field such that $\iota_Z\alpha_g=1$ and $\iota_Z \, d\alpha_g =0$. Equivalently, $Z$ is the unique vector field such that
\begin{equation*}
[X,Z]\in D,\quad [Y,Z]\in D,\quad [X,Y]=-Z\mod D ,
\end{equation*}
for any positive orthonormal local frame $(X,Y)$ of $D$.
The \emph{Reeb flow} $\mathcal{R}_t$ is defined as the flow generated on $M$ by the vector field $Z$.

\begin{remark}
The measure $\nu$ is invariant under the Reeb flow $\mathcal{R}_t$. This invariance property is important in view of identifying a QE result.
\end{remark}

\begin{remark}
Denoting by $h_Z$ the Hamiltonian associated with $Z$, geodesics with highfrequencies in $h_Z$ oscillate around the trajectories of $\pm Z$. From the point of view of semiclassical analysis, this part of the dynamics is expected to be the dominant one.
\end{remark}

\subsection{The main results}
We consider the cotangent space $(T^\star M,\omega)$, endowed with its canonical symplectic form.
Let $\Sigma \subset T^\star M$ be the characteristic manifold of $-\triangle_{sR}$. We have $\Sigma=(g^\star)^{-1}(0) = D^\perp$ (annihilator of $D$), and $\Sigma$ coincides with the cone of contact forms defining $D$. In particular, we have
\begin{equation*}
\Sigma = \{ (q,s\alpha_g(q))\in T^\star M \mid q\in M, s\in\R \}.
\end{equation*}
In the 3D contact case, we have as well $\Sigma = \{h_X=h_Y=0\}$ if $D=\mathrm{Span}(X,Y)$, with $(X,Y)$ a local $g$-orthonormal frame of $D$.

\paragraph{Shnirelman theorem in the 3D contact case.}
We are now in a position to provide a more precise statement of Theorem \ref{thm1}.

\begin{theorem}\label{theo:main} 
Let $M$ be a smooth connected compact three-dimensional manifold, equipped with an arbitrary smooth volume form $d\mu$. Let $D\subset TM$ be a smooth oriented contact subbundle, and let $g$ be a smooth Riemannian metric on $D$. 
Let $\triangle_{sR}$ be the sR Laplacian associated with the contact sub-Riemannian structure $(M,D,g)$ and with the volume form $d\mu$.

We assume that the Reeb flow is \emph{ergodic}\footnote{The Reeb flow is ergodic on $(M,\nu)$ if any measurable invariant subset of $M$ is of measure $0$ or $1$.}.
Then, for any \emph{real-valued} orthonormal Hilbert basis $(\phi_n)_{n\in \N^*}$ of $L^2(M,\mu)$ consisting of eigenfunctions of $\triangle_{sR}$
 associated with the eigenvalues $(\lambda_n)_{n\in\N^*}$, there exists a density-one sequence $(n_j)_{j\in\N^*}$ of positive integers such that
\begin{equation*}
\lim_{j\rightarrow+\infty} \left\langle A\phi_{n_j}, \phi_{n_j}\right\rangle  =  \frac{1}{2}\int_M \big( a(q,\alpha_g(q)) + a(q,-\alpha_g(q)) \big) \, d\nu ,
\end{equation*}
for every classical pseudo-differential operator $A$ of order $0$ with principal symbol $a$.

In particular, we have
$$
\lim_{j\rightarrow+\infty} \int_M f |\phi_{n_j}|^2\, d\mu = \int_M f\, d\nu ,
$$
for every continuous function $f$ on $M$.
\end{theorem}

Here, the notation $\langle\ ,\ \rangle$ stands for the (Hermitian) inner product in $L^2(M,\mu)$.

We prove in \cite{CHT-I} that the above result still holds with a complex-valued eigenbasis, provided the principal symbol $a$ satisfies $a(q,\alpha_g(q)) = a(q,-\alpha_g(q))$, for every $q\in M$. We also prove that the result is valid as well if $D$ is not orientable, without any additional assumption (note that, in that case, the contact form and thus the Reeb vector field $Z$ are not defined globally).

Note that the subsequence of density one depends on the choice of the eigenbasis $(\phi_n)_{n\in\N^*}$.
We stress that our result is valid for any choice of a smooth volume form $d\mu$ on $M$.

\begin{remark}
The classical Shnirelman theorem is established in the Riemannian setting under the assumption that the Riemannian geodesic flow is ergodic on $(S^\star M,\lambda_L)$, where the limit measure is the Liouville measure $\lambda_L$ on the unit cotangent bundle $S^\star M$ of $M$. In contrast, here the Liouville measure on the unit bundle $g^\star=1$ has infinite total mass (where $g^\star$ is the co-metric on $T^\star M $ associated with $g$), and hence the QE property cannot be formulated in terms of the geodesic flow.

Another interesting difference is that, in the Riemannian setting, QE says that most eigenfunctions equidistribute in the phase space, whereas here, in the 3D contact case, they concentrate on $\Sigma=D^\perp$, the contact cone that is the characteristic manifold of $\triangle_{sR}$.
\begin{table}[h]
\begin{center}
\begin{tabular}[width=1mm]{| l || c | c |}
\hline
& Riemannian case & SR case \\
\hline\hline 
Ergodicity assumption & geodesic flow on $(S^*M,\textrm{Liouville})$ & Reeb flow on $(M,\textrm{Popp})$ \\
\hline
Quantum limit & Liouville measure on $S^*M$ & Popp measure on $M$ \\
\hline
Microlocal concentration & on $S^*M$ & on $S\Sigma=\Sigma\cap\{\vert h_Z\vert=1\}$\\
\hline
\end{tabular}
\end{center}
\end{table}
\end{remark}

\begin{remark}[Examples of ergodic Reeb flows in dimension $3$]
Let us give two general constructions providing examples of ergodic Reeb flows on 3D contact manifolds:
\begin{itemize}
\item Geodesic flows:
Let $X$ be a two-dimensional compact Riemannian surface, endowed with a Riemannian metric $h$, and let $M=S^\star X$ be the unit cotangent
 bundle of $X$.
The closed three-dimensional manifold $M$ is then naturally endowed with the contact form $\alpha$ defined as the restriction to $M$ of the Liouville 1-form $\lambda= p\, dq$. Let $Z$ be the associated Reeb vector field.
Identifying the tangent and cotangent bundles of $X$ with the metric $h$, the set $M$ is viewed as the unit tangent bundle of $X$. Then $Z$ is identified with the vector field on the unit tangent bundle of $X$ generating the geodesic flow on $S^\star X$. Therefore, with this identification, the Reeb flow is the geodesic flow on $M$.

This geodesic flow is ergodic for instance if the curvature of $X$ is negative. 

\item Hamiltonian flows:
Let $(W,\omega)$ be a symplectic manifold of dimension $4$, and let $M$ be a submanifold of $W$ of dimension $3$, such that there exists a vector field $v$ on a neighborhood of $M$ in $W$, satisfying $\mathcal{L}_v\omega=\omega$ (Liouville vector field), and transverse to $M$. Then the one-form $\alpha=\iota_v\omega$ is a global contact form on $M$, and we have $d\alpha=\omega$. Note that, if $\omega=d\lambda$ is exact, then the vector field $v$ defined by $\iota_v\omega=\lambda$ is Liouville (in local symplectic coordinates $(q,p)$ on $W$, we have $v=p\,\partial_p$). If the manifold $M$ is moreover a level set of an Hamiltonian function $h$ on $W$, then the Reeb flow on $M$ (associated with $\alpha$) is a reparametrization of the Hamiltonian flow restricted to $M$.

If $D=\ker\alpha$ is moreover endowed with a Riemannian metric $g$,
then $\alpha_g = h \alpha$ for some smooth function $h$ (never vanishing). 
Let us then choose the metric $g$ such that $h=1$. 
Then, the Reeb flow is ergodic on $(M,\nu)$ if and only if the Hamiltonian flow is ergodic on $(W,\omega^2)$.
\end{itemize}
\end{remark}

\paragraph{Quantum limits in the 3D contact case.}
Let $(\psi_j)_{j\in\N^*}$ be an arbitrary orthonormal family of $L^2(M,\mu)$. We set $\mu_j(a) = \langle \Op(a)\psi_j,\psi_j\rangle$, for every $j\in\N^*$, and for every classical symbol $a$ of order $0$. The measure $\mu_j$ is asymptotically positive, and any closure point (weak limit) of $(\mu_j)_{j\in\N^*}$ is a probability measure on the sphere bundle $S^\star M$, called a \emph{quantum limit} (QL), or a \emph{semi-classical measure},  associated with the family $(\psi_j)_{j\in\N^*}$.

Theorem \ref{thm1} says that, under the ergodicity assumption of the Reeb flow, the probability Popp measure $\nu$, which is invariant under the Reeb flow, is the ``main" quantum limit associated with any eigenbasis.

The following result provides an insight on quantum limits of eigenfunctions in the 3D contact case in greater generality, without any ergodicity assumption.
In order to state it, we identify $S^\star M = (T^\star M\setminus\{0\})/(0,+\infty)$ with the union of the unit cotangent bundle $U^\star M =\{ g^\star =1 \}$ and of the sphere bundle $S\Sigma = (\Sigma\setminus\{0\})/(0,+\infty)$ which is a two-fold covering of $M$.
Each fiber is obtained by compactifying a cylinder with two points at infinity.
Moreover, the Reeb flow can be lifted to $S\Sigma $.

\begin{theorem}\label{thm2}
Let $(\phi_n)_{n\in \N^*}$ be an orthonormal Hilbert basis of $L^2(M,\mu)$, consisting of eigenfunctions of $\triangle_{sR}$ associated with the eigenvalues $(\lambda_n)_{n\in\N^*}$.
\begin{enumerate}
\item Let $\beta$ be a QL associated with the family $(\phi_n)_{n\in\N^*}$. Using the above identification $S^\star M = U^\star M \cup S\Sigma$, the measure $\beta$ can be identified to the sum $\beta =\beta_0 + \beta _\infty $ of two mutually singular measures such that:
\begin{itemize}
\item $\beta_0$ is supported on $U^\star M$ and is invariant under the sR geodesic flow associated with the sR metric $g$,
\item $\beta_\infty$ is supported on $S\Sigma $ and is invariant under the lift of the Reeb flow.
\end{itemize} 
\item There exists a density-one sequence $(n_j)_{j\in\N^*}$ of positive integers such that, if $\beta$ is a QL associated with the orthonormal family $(\phi_{n_j})_{j\in \N^*}$, then the support of $\beta$ is contained in $S\Sigma$, i.e., $\beta _0=0$ in the previous decomposition.
\end{enumerate}
\end{theorem}

The proof of that result follows from arguments used to prove Theorem \ref{thm1} (see \cite{CHT-I} for details).

\subsection{The model example: compact Heisenberg flat case}
The simplest example is given by an invariant metric on a compact quotient of the Heisenberg group. The spectral decomposition of the Heisenberg Laplacian is then explicit and can serve as a model to derive our main result.

Let $G$ be the three-dimensional Heisenberg group defined as $G=\R^3$ with the product rule
$$
(x,y,z)\star (x',y',z')=(x+x',y+y',z+z'-xy') .
$$
The contact form $\alpha_H = dz+x\,dy $ and the vector fields $X_H= \partial_x$ and $Y_H= \partial_y - x\partial_z $ are left-invariant on $G$.
Defining the discrete co-compact subgroup $\Gamma$ of $G$ by $\Gamma =\{ (x,y,z)\in G \ \mid\ x,y\in \sqrt{2\pi}\Z,\, z\in 2\pi  \Z \}$, we then define the three-dimensional compact manifold manifold $M_H = \Gamma \backslash  G$, and we consider the horizontal distribution $D_H=\ker \alpha_H$, endowed with the metric $g_H$ such that $(X_H,Y_H)$ is a $g_H$-orthonormal frame of $D_H$. With this choice, we have ${\alpha_H}_g=\alpha_H$.

The Reeb vector field is given by $Z_H=-[X_H,Y_H]=\partial_z $.
The Lebesgue volume $d\mu = dx\, dy\, dz$ coincides with the Popp volume $dP$, and we consider the corresponding sub-Riemannian Laplacian
 $\triangle_H=X_H^2+Y_H^2$ (here, the vector fields have divergence zero).

We refer to this sub-Riemannian case as the \textit{Heisenberg flat case} $(M_H,\triangle_H)$.
The cotangent space $T^\star M_H$ is endowed with its canonical symplectic form.

\medskip

It is proved in \cite{yCdV-83} that the spectrum of $-\triangle_H$ is given by
$$
\{\lambda _{\ell,m}= (2\ell+1)|m| \mid m\in \Z\setminus \{0\},\, \ell\in\N \}
\cup \{\mu_{j,k}= {2\pi }(j^2 + k^2) \mid j,k\in\Z  \} ,
$$
where $ \lambda _{\ell,m} $ is of multiplicity $|m| $. 
From this, we infer two things.

The first (already known, see, e.g., \cite{Metivier1976}) is that the spectral counting function has the asymptotics
\begin{equation*} 
N(\lambda )\sim \sum_{l=0}^\infty  \frac{\lambda^2}{(2l+1)^2} = \frac{\pi^2}{8}\lambda^2 = \frac{P(M_H)}{32}\lambda^2,
\end{equation*} 
as $\lambda\rightarrow+\infty$.
It is interesting to compare that result with the corresponding result for a Riemannian Laplacian on a three-dimensional closed Riemannian manifold, for which $N(\lambda )\sim C\lambda^{3/2}$.

The second thing is the following.
Microlocally, in the cone $C_c=\{ p_z^2 > c (p_x^2 +p_y^2) \}$ for some $c>0$ (with local coordinates $p=(p_x,p_y,p_z)$ in the cotangent space), the sub-Riemannian Laplacian can be written as a commuting product
$$
\boxed{-\triangle_H = R_H\Omega_H = \Omega_H R_H},
$$
with
\begin{itemize}
\item $R_H=\sqrt{Z_H^\star Z_H}$ (pseudo-differential operator of order $1$),
\item $\Omega_H =U_H^2+ V_H^2$ (harmonic oscillator), where $U_H$ and $V_H$ are the pseudo-differential operators of order $1/2$ defined by 
$$
U_H = \frac{1}{i}R_H^{-\frac{1}{2}}X_H = \OpW\left( \frac{h_{X_H}}{\sqrt{\vert h_{Z_H}\vert}} \right), \qquad
V_H = \frac{1}{i}R_H^{-\frac{1}{2}}Y_H = \OpW\left( \frac{h_{Y_H}}{\sqrt{\vert h_{Z_H}\vert}} \right),
$$
where the notation $\OpW$ stands for the Weyl quantization.
Note that $[U_H,V_H]= \pm\mathrm{id}$ (according to the sign of $h_{Z_H}$) and $\exp(2i\pi\Omega_H)=\mathrm{id}$.
\end{itemize}
In terms of symbols, since $X_H$ and $Y_H$ have divergence zero, we have factorized $\sigma(-\triangle_H) = h_{X_H}^2+h_{Y_H}^2$ (full symbol) as
$
\vert h_{Z_H} \vert \left( \left( \frac{h_{X_H}}{\sqrt{\vert h_{Z_H}\vert }} \right)^2 + \left( \frac{h_{Y_H}}{\sqrt{\vert h_{Z_H}\vert }} \right)^2 \right) .
$

\section{The main steps of the proof of Theorem \ref{theo:main}}\label{sec3}
\subsection{General concepts}\label{sec31}
We fix an arbitrary (orthonormal) eigenbasis $(\phi_n)_{n\in\N^*}$ of $\triangle_{sR}$. We recall that the spectral counting function is defined by $N(\lambda)=\# \{n \mid \lambda_n \leq \lambda \}$.

\begin{definition}
For every bounded linear operator $A$ on $L^2(M,\mu)$, we define the Ces\`aro mean $E(A)\in \C$ by
\begin{equation*}
E(A) = \lim_{\lambda \rightarrow +\infty} \frac{1}{N(\lambda)}\sum_{\lambda_n \leq \lambda } \left\langle A \phi_n , \phi_n \right\rangle,
\end{equation*}
whenever this limit exists.
We define the variance $V(A)\in \R^+$ by
\begin{equation*}
V(A) = \limsup _{\lambda \rightarrow +\infty } \frac{1}{N(\lambda)}\sum _{\lambda_n \leq \lambda} \left\vert\left\langle A \phi_n ,\phi_n \right\rangle \right\vert^2 .
\end{equation*}
\end{definition}

It is easy to see that $V(A)\leq E(A^\star A)$, and that, if $A$ is a compact operator on $L^2(M,\mu)$, then $E(A)=0$ and $V(A)=0$.

\begin{definition}\label{defi:weyl-measure}
The \emph{local Weyl measure} $w_\triangle $ is the probability measure on $M$ defined by
$$
\int _M f \, dw_\triangle = \lim _{\lambda \rightarrow +\infty }\frac{1}{N(\lambda)}\sum _{\lambda_n \leq \lambda } \int_M f |\phi_n |^2 \, d\mu  ,
$$
for every continuous function $f:M\rightarrow \R $, whenever the limit exists.
In other words, 
$$
w_\triangle = \mathrm{weak}\lim_{\lambda \rightarrow +\infty} \frac{1}{N(\lambda)}\sum _{\lambda_n \leq \lambda } |\phi_n |^2 \mu .
$$
The \emph{microlocal  Weyl measure} $W_\triangle $ is the probability measure on $S^\star M=S(T^\star M)$, the co-sphere bundle, defined as follows:
we identify positively homogeneous functions of degree $0$ on $T^\star M$ with functions on the bundle $ S^\star M$. Then, for every symbol $a:S^\star M \rightarrow \R $ of order zero, we have
$$
\int _{S^\star M} a \, dW_\triangle = \lim _{\lambda \rightarrow +\infty }\frac{1}{N(\lambda)}\sum _{\lambda_n \leq \lambda } \langle \Op_+ (a)\phi_n , \phi_n \rangle   ,
$$
where $\Op_+$ is a positive quantization, whenever the limit exists.
In other words, $W_\triangle $ is the weak limit of the probability measures on $S^\star M $ defined by 
$$
a \mapsto \frac{1}{N(\lambda)}\sum _{\lambda_n \leq \lambda } \langle \Op_+ (a)\phi_n , \phi_n \rangle .
$$
\end{definition}

The microlocal Weyl measure, if it exists, does not depend on the choice of the quantization and of the choice of the orthonormal eigenbasis (because it is a trace). Moreover, $W_\triangle $ is even with respect to the canonical involution of $S^\star M$.

\paragraph{General path towards QE.}
In order to establish QE, we follow the general path of proof described in \cite{Ze-10}. Defining the spectral counting function by $N(\lambda)=\# \{n \mid \lambda_n \leq \lambda \}$, the first step consists in establishing a \textit{microlocal Weyl law}:
\begin{equation}\label{localWeyllaw}
E(A) = \int _{ST^\star M } a\, dW_\triangle,
\end{equation}
for every classical pseudo-differential operator $A$ of order $0$ with a principal symbol $a$, where $S T^\star M$ is the unit sphere bundle of $T^\star M$.
This Ces\'aro convergence property can usually be established under weak assumptions, without any ergodicity property.

The second step consists in proving the \textit{variance estimate}
\begin{equation}\label{varianceestimate}
V (A- E(A) \mathrm{id} ) =0 .
\end{equation}
The variance estimate usually follows by combining the microlocal Weyl law \eqref{localWeyllaw} with ergodicity properties of some
 associated classical dynamics and with an Egorov theorem.

Then QE follows from the two properties above. Indeed, for a fixed pseudo-differential operator $A$ of order $0$, it follows from
 \eqref{varianceestimate} and from a well known lemma\footnote{This lemma states that, given a bounded sequence $(u_n)_{n\in\N}$ of nonnegative
 real numbers, the Ces\'aro mean $\frac{1}{n}\sum_{k=0}^{n-1}u_k$ converges to $0$ if and only if there exists a subset $S\subset\N$ of density
 one such that $(u_k)_{k\in S}$ converges to $0$. We recall that $S$ is of density one if $\frac{1}{n}\#\{k\in S\mid k\leq n-1\}$ converges to $1$
 as $n$ tends to $+\infty$.} due to Koopman and Von Neumann (see, e.g., \cite[Chapter 2.6, Lemma 6.2]{Petersen}) that there exists a density-one
 sequence $(n_j)_{j\in\N^*}$ of positive integers such that $\left\langle A \phi_{n_j}, \phi_{n_j} \right\rangle \rightarrow E(A)$ as $j\rightarrow +\infty$.
Using the fact that the space of symbols of order $0$ admits a countable dense subset, QE is then established with a diagonal argument.

\subsection{The microlocal Weyl law}\label{sec32}
In order to prove Theorem \ref{theo:main}, we follow the general path above, and we start by providing a microlocal Weyl law, identifying the microlocal Weyl measure in the 3D contact case.

\begin{theorem}\label{theo:weyl}
Let $A$ be a classical pseudo-differential operator of order $0$ with principal symbol $a$. We have
\begin{equation*}
\sum_{\lambda_n \leq \lambda } \left\langle A \phi_n , \phi_n \right\rangle =
 \frac{P(M)}{64} \lambda^2 (1+\mathrm{o}(1)) \int_M \big( a(q,\alpha_g (q))  +  a(q,-\alpha_g (q)) \big) \, d\nu  ,
\end{equation*}
as $\lambda\rightarrow+\infty$.
In particular, it follows that $N(\lambda) \sim \frac{P(M)}{32}\lambda^2$,
and that 
\begin{equation*}
E(A) = \frac{1}{2}  \int_M \big( a(q,\alpha_g (q)) + a(q,-\alpha_g (q)) \big) \, d\nu = \int_{S\Sigma} a \, dW_\triangle ,
\end{equation*}
where $S\Sigma = (\Sigma\setminus\{0\})/(0,+\infty)$ is the co-sphere bundle of $\Sigma$.
\end{theorem}

Our proof of Theorem \ref{theo:weyl}, done in \cite{CHT-I}, consists in first establishing a local Weyl law, by computing heat traces with the sR heat kernel and using the Karamata tauberian theorem, and then in infering from that local law, the microlocal one. The latter step is performed as follows.
We first prove that, if the microlocal Weyl measure $W_\triangle $ exists, then $\mathrm{Supp}(W_\triangle)\subset S \Sigma $. This follows from the fact that, outside of the characteristic manifold $\Sigma$, the operator $\triangle_{sR}$ is elliptic, and therefore classical arguments of \cite{DG-75,Ho-68}, using Fourier Integral Operators and wave propagation, yield that
$$
\sum _{\lambda_n \leq \lambda }| \langle A \phi_n ,\phi_n \rangle | = \mathrm{O}\big( \lambda^{3/2} \big),
$$
as $\lambda\rightarrow+\infty$, for every pseudo-differential operator of order $0$ whose principal symbol vanishes in a neighborhood of $\Sigma$. Since we already know that $N(\lambda) \sim C \lambda^2$, the concentration on $\Sigma$ follows.
It can be noted that these arguments are as well valid for general sub-Riemannian structures, not only in the 3D contact case.

Then, we prove that, since the horizontal distribution $D$ is of codimension $1$ in $TM$, and since we have already identified the local Weyl measure $w_\triangle$, then the microlocal Weyl measure $W_\triangle$ exists and is equal to half of the pullback of $w_\triangle$ by the double covering $S \Sigma \rightarrow M $ which is the restriction of the canonical projection of $T^\star M $ onto $M$.
This latter argument is also valid in a more general context, as soon as $D$ is of codimension $1$.

\begin{remark}\label{rk:changemu1}
A remark which is useful in order to understand why our result does not depend on the choice of the measure is the following. Let us consider two sR Laplacians $\triangle_{\mu_1}$ and $\triangle_{\mu_2}$ associated with two different volume forms (but with the same metric $g$).
We assume that $\mu_2= h^2\mu_1$ with $h$ a positive smooth function on $M$.
We define the isometric bijection $J:L^2(M,\mu_2)\rightarrow L^2(M,\mu_1)$ by $J\phi=h\phi$. Then $
J\triangle_{\mu_2}J^{-1} = \triangle_{\mu_1} + h \triangle_{\mu_2}(h^{-1})  \,\mathrm{id} $, and therefore $\triangle_{\mu_1}$ is unitarily equivalent to $\triangle_{\mu_2} +W$, where $W$ is a bounded operator.
\end{remark}

\subsection{Birkhoff normal form}
Having in mind the model situation given by the Heisenberg flat case $(M_H,\triangle_H)$, in the general 3D contact case we are able to establish a Birkhoff normal form, in the spirit of a result by Melrose in \cite[Section 2]{Me-84}, which implies in particular that, microlocally near the characteristic cone, all 3D contact sub-Riemannian Laplacians (associated with different metrics and/or measures) are equivalent. 

We define the positive conic submanifolds $\Sigma^\pm = \{ (q,s\alpha_g(q))\in T^\star M \mid \pm s>0 \}$ of $T^\star M$.
In the Heisenberg flat case, the characteristic cones are defined accordingly.

Given $k\in \N\cup \{+\infty\}$ and given a smooth function $f$ on $T^\star M$, the notation $f=\mathrm{O}_\Sigma(k)$
 means that $f$ vanishes along $\Sigma$ at order $k$ (at least). The word flat is used when $k=+\infty$.

\paragraph{Classical normal form.}

\begin{theorem}\label{thm:BNF}
Let $q\in M$ be arbitrary. There exist a conic neighborhood $C_q$ of $\Sigma^{+}_q$ in $(T^\star M,\omega)$ and a homogeneous symplectomorphism $\chi$ from $C_q$ to $(T^\star M_H,\omega_H)$, satisfying $\chi(q)=0$ and
 $\chi(\Sigma^{+}\cap C_q)\subset\Sigma_H^+$, such that 
\begin{equation*}
\boxed{\sigma_P(-\triangle_{H})\circ \chi =\sigma_P(-\triangle_{sR})+\mathrm{O}_\Sigma(\infty)}
\end{equation*}
\end{theorem}

Of course, a similar result can be given for $\Sigma^-$.

In order to establish this normal form, we first endow $\R^6$ with a symplectic form $\tilde\omega$,
 with an appropriate conic structure, and with an Hamiltonian function $H_2$, such that, for any given contact
 structure and any $q\in M$, there exists a homogeneous diffeomorphism $\chi_1$ from a conic neighborhood $C_q$ of $\Sigma_q^+$ to $\R^6$ such that $\chi_1^* \tilde{\omega}=\omega+\mathrm{O}_\Sigma(1)$ and $H_2\circ \chi_1  = \sigma_P(-\triangle_{sR})$. Thanks to the Darboux-Weinstein lemma, we modify $\chi_1$ into a homogeneous diffeomorphism $\chi_2$ such that $\chi_2^* \tilde{\omega}=\omega$ and $H_2\circ \chi_2 = \sigma_P(-\triangle_{sR})+\mathrm{O}_\Sigma(3)$.
Finally, we improve the latter remainder to a flat remainder $\mathrm{O}_\Sigma(\infty)$, by solving an infinite number of cohomological equations in the symplectic conic manifold $(\R^6,\tilde{\omega})$.

\paragraph{Quantum normal form.}
By quantizing the above Birkhoff normal form, we obtain the following result.

\begin{theorem}\label{thm_quantum_normal_form}
For every $q\in M$, there exists a (conic) microlocal neighborhood $\tilde{U}$ of $\Sigma_q$ in $T^\star M$ such that, considering all the following pseudo-differential operators as acting on functions microlocally supported in $\tilde{U}$, we have
\begin{equation*}
\boxed{ -\triangle_{sR} = R\Omega + V_0+\mathrm{O}_\Sigma(\infty) }
\end{equation*}
where
\begin{itemize}
\item $V_0\in \Psi^0$ is a self-adjoint pseudo-differential operator of order $0$,
\item $R$ and $\Omega$ are self-adjoint pseudo-differential operators of order $1$, of respective principal symbols satisfying
$$
\sigma_P(R)=\vert h_Z\vert + \mathrm{O}_\Sigma(2),
\qquad
\sigma_P(\Omega)=(h_X^2+h_Y^2)/\vert h_Z\vert + \mathrm{O}_\Sigma(3),
$$
\item $[R,\Omega] = 0 \mod \Psi^{-\infty}$ ,
\item $\exp(2i\pi\Omega)=\mathrm{id} \mod \Psi^{-\infty}$.
\end{itemize}
\end{theorem}

The latter remainders are in the sense of pseudo-differential operators of order $-\infty$.
In the flat Heisenberg case, there are no remainder terms, and we recover the operators $R_H$ and $\Omega_H$. The pseudo-differential operators $R$ and $\Omega$ can be seen as appropriate perturbations of $R_H$ and $\Omega_H$, designed such that the last two items of Theorem \ref{thm_quantum_normal_form} are satisfied.

Note that the operators $R$ and $\Omega$ depend on the microlocal neighborhood $\tilde U$ under consideration. This neighborhood can then be understood as a chart in the manifold $T^\star M$, in which the quantum normal form is valid.

\subsection{Sketch of proof of the variance estimate in the Heisenberg flat case}
As explained at the end of Section \ref{sec31}, in order to establish QE, it suffices to prove that $V(A)=0$ for every pseudo-differential operator $A$ of order $0$ satisfying $E(A)=0$.

We provide the proof of that fact in the Heisenberg flat case, that is, when $-\triangle_{sR}=R_H\Omega_H=\Omega_H R_H$. For simplicity of notation, we drop the index $H$ in what follows.

\begin{remark}
In the Heisenberg flat case, the microlocal factorization of $-\triangle_{sR}$ is global. This fact avoids many technicalities. In \cite{CHT-I} where the complete proof is done, we have to use microlocal charts where the quantum normal form of Theorem \ref{thm_quantum_normal_form} is valid. We have also to take care of remainder terms, since in the general case we do not have an exact factorization, and handling the additional terms raises additional difficulties.
\end{remark}

Let $A$ be a pseudo-differential operator of order $0$ satisfying $E(A)=0$. Let $a=\sigma_P(A)$ be its principal symbol.

\paragraph{Preparation of $A$.} We claim that we can modify $A$ without changing $a_{\vert\Sigma}$, so as to assume that $[A,\Omega]=0$.

Indeed, it suffices to average with respect to $\Omega$, as follows. Setting $A_s = \exp(is\Omega) A \exp(-is\Omega)$, for $s\in\R$, we have, by the Egorov theorem, $\sigma_P(A_s) = a\circ \exp (t\vec\sigma_P(\Omega))$. Since $\sigma_P(\Omega)=(h_X^2+h_Y^2)/\vert h_Z\vert=0$ along $\Sigma$, it follows that $\sigma_P(A_s)_{\vert\Sigma}=a$.
Now, setting $\bar A = \frac{1}{2\pi}\int_0^{2\pi} A_s\, ds$, we have $\sigma_P(\bar A)_{\vert\Sigma}=a$. Using that $\exp(2i\pi\Omega)=\mathrm{id}$ and that $\frac{d}{ds} A_s=i[\Omega,A_s]$, we infer that $[\Omega,\bar A]=0$. Then we replace $A$ with $\bar A$.

\paragraph{Averaging with respect to $R$.} Let us now set $A_t = \exp(-itR)A\exp(itR)$, for $t\in\R$. By the Egorov theorem, we have $a_t=\sigma_P(A_t)=a\circ\exp(t\vec h_Z)$, where $\vec h_Z$ is the Hamiltonian vector field generated by the Hamiltonian function $h_Z$.
For $T>0$, we define
$\bar A_T = \frac{1}{T}\int_0^T A_t\, dt$.
To prove that $V(A)=0$, it suffices to prove that
\begin{equation}\label{eq1}
V(A-\bar A_T)=0,
\end{equation}
and that
\begin{equation}\label{eq2}
\lim_{T\rightarrow +\infty}V(\bar A_T)=0 .
\end{equation}
The next two lemmas are devoted to prove \eqref{eq1} and \eqref{eq2}.
Ergodicity will be used to prove \eqref{eq2}.

\begin{lemma}
We have $V(A-A_t)=0$, for every $t\in\R$. As a consequence, we have $V(A-\bar A_T)=0$, for every $T>0$.
\end{lemma}

\begin{proof}
Since $\frac{d}{dt} A_t=i[A_t,R]$, it suffices (by Cauchy-Schwarz) to prove that $V([A_t,R])=0$.
Using the Jacobi identity, and the fact that $[R,\Omega]=0$, we compute
\begin{equation*}
\frac{d}{dt} [A_t,\Omega] =  i[[A_t,R],\Omega] 
=  i [R,[\Omega,A_t]]+i[A_t,[R,\Omega]]
= -i \mathrm{ad}(R).[A_t,\Omega] ,
\end{equation*}
and since $[A,\Omega]=0$ by construction, we infer that $[A_t,\Omega] = e^{-it\,\mathrm{ad}(R)}[A,\Omega] =0$.

Now, for every integer $n$, we have
\begin{equation*}
\begin{split}
\langle [A_t,R]\phi_n,\phi_n\rangle
&= \langle A_tR\phi_n,\phi_n\rangle - \langle RA_t\phi_n,\phi_n\rangle \\
&= -\frac{1}{\lambda_n}\langle A_tR\phi_n,R\Omega\phi_n\rangle + \frac{1}{\lambda_n}\langle RA_tR\Omega\phi_n,\phi_n\rangle \qquad\textrm{because}\ \phi_n = -\frac{1}{\lambda_n}\triangle_{sR}\phi_n \\
&= \frac{1}{\lambda_n}\langle [RA_tR,\Omega]\phi_n,\phi_n\rangle\\
&= 0
\end{split}
\end{equation*}
because $[R,\Omega]=0$ and $[A_t,\Omega]=0$.
Hence $V(A-A_t)=0$ for every $t$. 

Using the Fubini theorem and the Jensen inequality, we easily infer that $V(A-\bar A_T)=0$.
\end{proof}

\begin{lemma}
We have $\displaystyle\lim_{T\rightarrow +\infty}V(\bar A_T)=0$.
\end{lemma}

\begin{proof}
We have $V(\bar A_T)\leq E(\bar A_T\bar A_T^*)$, with $\sigma_P(\bar A_T\bar A_T^*) = \vert a_T\vert^2 = \vert\frac{1}{T}\int_0^Ta_t\, dt\vert^2$.
By the microlocal Weyl law (Theorem \ref{theo:weyl}), we have
$$
E(\bar A_T\bar A_T^*) = \frac{1}{2} \sum_{\pm} \int_M \vert a_T(q,\pm\alpha_g(q))\vert^2\, d\nu ,
$$
but, using the Reeb flow $\mathcal{R}_t$,
\vspace{-2mm}
$$
f^\pm_T(q) = a_T(q,\pm\alpha_g(q)) = \frac{1}{T}\int_0^T a_t(q,\pm\alpha_g(q))\, dt = \frac{1}{T}\int_0^T a(\mathcal{R}_t(q),\pm\alpha_g(\mathcal{R}_t(q)))\, dt
$$
By ergodicity of the Reeb flow, and using the von Neumann ergodic theorem, we have
$$
f^\pm_T\underset{T\rightarrow +\infty}{\overset{L^2}{\longrightarrow}}\int_M a(q,\pm\alpha_g(q))\, d\nu=2E(A)=0,
$$
and the result follows, under the additional assumption that $a$ is even with respect to $\Sigma$. If the eigenfunctions are real-valued then this assumption can actually be dropped (see \cite{CHT-I}).
\end{proof}

\section{Perspectives and open problems}\label{sec4}


\paragraph{Reeb flow and spectral theory.}
In addition to QE, there are several other relationships between the Reeb flow and the spectral theory of sR Laplacians in the 3D contact case.
For example, in the case where $M=S^3$, the Popp volume is a spectral invariant, which coincides with the inverse of the Arnold asymptotic linking number of the Reeb orbits. 
Moreover, using a global normal form, one can figure out a close relationship between the spectrum of an elliptic Toeplitz operator that is the quantization of the Reeb flow and the spectrum of $\triangle_{sR}$.
These arguments lead to conjecture that the periods of the periodic orbits of the Reeb flow are spectral invariants of $\triangle_{sR}$.

\paragraph{Measures in sR geometry.}
A general problem consists of showing the existence of the (micro)-local Weyl measures (defined in Definition \ref{defi:weyl-measure}), and of identifying it. 

In contrast to Riemannian geometry, in sR geometry there are several possible choices of intrinsic volume forms, for instance: the standard Hausdorff volume (which is defined with arbitrary coverings), the spherical Hausdorff volume (which is defined with ball coverings only, with balls of constant radius), the Popp volume. The latter one, which is always a smooth volume, was introduced and defined in the equiregular case in \cite{Mo-02}.
Note that for left-invariant (Lie group) sub-Riemannian manifolds, the Popp measure and the standard and spherical Hausdorff measures are left-invariant, and thus are proportional to the left-Haar measure. In other words they coincide up to a multiplying scalar.
In particular in the 3D contact case, the Popp measure and the standard and spherical Hausdorff measures are proportional, but the constant is not equal to $1$, and is not known.

All these measures do not coincide in general (if there is no group structure). More precisely, it is proved in \cite{AgrachevBarilariBoscain_CVPDE2012} that the Popp measure and the spherical Hausdorff measure are proportional for regular sR manifolds of dimensions $3$ and $4$, but in dimension larger than or equal to $5$, the spherical Hausdorff measure is not smooth in general and therefore differs from the Popp measure (which is always smooth by definition for equiregular distributions).

The notion of (microlocal, or local) Weyl measure is very natural in the spectral context, but seems to be new and has never been studied in sR geometry.
It is therefore of interest to compare it with other notions of measures and to investigate the following kind of question: does the Weyl measure coincide with an appropriate Hausdorff measure? What is its Radon-Nikodym derivative with respect to the Popp measure?

\paragraph{Equiregular sR structures.}
In Theorem \ref{theo:weyl}, we have obtained a microlocal Weyl law in the 3D contact case.
It is likely that the microlocal Weyl formula can be generalized to equiregular sub-Riemannian structures.

Let $(M,D,g)$ be a sub-Riemannian structure, where $M$ is a compact manifold of dimension $d$. Let $\triangle _{sR}$ be a sub-Riemannian Laplacian on $M$. We assume that $\mathrm{Lie}(D)=TM$ (H\"ormander's assumption, implying hypoellipticity).
We say that the horizontal distribution $D$ is \emph{equiregular}, if the sequence of subbundles $(D_i)_{i\in\N^*}$ defined by $D_0=\{0\}$, $D_1=D$, $D_{i+1}=D_i+[D_i,D]$, is such that the dimension of ${D_i}_q$ dot not depend on $q\in M$, for every integer $i$. Let $r$ be the smallest positive integer such that $D_r=TM$.
We set $\Sigma_j = D_j^\perp$, for $j=0,\ldots,r$.
\\
In the equiregular case, it is known (see \cite{Metivier1976}) that $N(\lambda)\sim \mathrm{Cst}\lambda^{Q/2}$, with $Q = \sum_{i=1}^r i\dim(D_i/D_{i-1})$ (Hausdorff dimension). 
Besides, as already said in Section \ref{sec32}, we already know that, if the microlocal Weyl measure exists, then $\mathrm{supp}(W_\triangle)\subset S\Sigma=S\Sigma_1$. Actually, we conjecture that
\begin{equation*}
\mathrm{supp}(W_\triangle) = S\Sigma_{r-1} ,
\end{equation*}
and that this result on the support is even valid in general, not only in the equiregular case.

\paragraph{Contact case in dimension $5$.} 
Theorem \ref{theo:main} has been established in the 3D contact case. It is natural to investigate the general contact case, however, already in the contact 5D case, the situation is much more complicated, and we cannot expect to have a Birkhoff normal form at the infinite order as in Theorem \ref{thm:BNF}. Indeed, in dimension $5$, the characteristic manifold $\Sigma$ is of codimension $4$, and we have two harmonic oscillators, which may have resonances according to the choice of the coefficients of the sR metric.
We expect however to be able, in that case, to derive a Birkhoff normal form at finite orders, and thanks to a quantized version of that normal form, to be able to derive QE under generic assumptions. Indeed, resonances should only occur on sets of measure zero in the generic case.

\paragraph{The role of singular curves.}
We briefly recall the definition of a singular curve. Let $(M,D,g)$ be a sR manifold. A curve $q(\cdot):[0,1]\rightarrow M$ is said to be \emph{horizontal} if it is absolutely continuous and if $\dot q(t)\in D_{q(t)}$ for almost every $t\in[0,1]$.
Let $q_0\in M$ be arbitrary. Let $\Omega(q_0)$ be the set of all horizontal curves $q(\cdot)$ such that $q(0)=q_0$. It is clear that $\Omega(q_0)$ is a Banach manifold, modeled on $L^1(0,1;\R^m)$ (with $m=\mathrm{rank}(D)$). Now, let $q_1\in M$ be another arbitrary point. The set $\Omega(q_0,q_1)$ of all $q(\cdot)\in \Omega(q_0)$ such that $q(1)=q_1$ may not be a manifold. Indeed, defining the end-point mapping $\mathrm{end}_{q_0}:\Omega(q_0)\rightarrow M$ by $\mathrm{end}_{q_0}(q(\cdot))=q(1)$, we have $\Omega(q_0,q_1) = (\mathrm{end}_{q_0})^{-1}(q_1)$, and the mapping $\mathrm{end}_{q_0}$ need not be a submersion. By definition, a singular curve $q(\cdot)$ is an horizontal curve that is a critical point of $\mathrm{end}_{q(0)}$, or, in other words, such that the differential $d\,\mathrm{end}_{q(0)}$ is not of full rank.

There are many clues leading one to think that singular curves may have a strong impact on the asymptotic spectral properties of the sR Laplacians. 
For instance, we expect that they have an impact on the Schwartz kernel in the representation through Feynman integrals. This representation has the form
$$
\exp(ith\triangle)f(x) = \int_{q(\cdot)\in\Omega(x)} e^{\frac{i}{h}\int_0^t L(q(s),\dot q(s))\, ds} f(q(t)) \, d\gamma(q(t)) ,
$$
and if we fix the terminal points of the paths, then we ``disintegrate" the measure to obtain the Schwartz kernel. We expect that singular curves cause a singularity in the kernel, which differs from the usual one.
In order to understand such issues, we provide hereafter several examples.

\paragraph{Martinet case.}
We consider a smooth compact connected manifold $M$ of dimension $3$, and a generic horizontal distribution $D$ of rank $2$, defined by $D=\ker\alpha$, with $\alpha$ a nontrivial one-form on $M$ such that $\alpha\wedge d\alpha$ vanishes transversally on a surface $S$ of dimension $2$. This surface is called the \emph{Martinet surface}.
Let $g$ be a Riemannian metric on $D$.

The local model in $\R^3$ near $0$ is given by $\alpha = dz - x^2\, dy$. Then $D$ is locally spanned by the two vector fields $X=\partial_x$ and $Y=\partial_y+x^2\partial_z$. Note that $S=\{x=0\}$ and that $[X,Y]=2x\partial_z$ vanishes along $S$. A Lie bracket of length two, $[X,[X,Y]]=2\partial_z$, is required to generate the missing direction, along $S$. Outside of $S$, the distribution $D$ is of contact type.

Note that the characteristic manifold $\Sigma$ is not symplectic. This is equivalent to the existence of nontrivial singular curves. The Martinet case is indeed a well known sR model in which there are nontrivial singular curves, that are minimizing, and locally foliate the Martinet surface $S$ (see \cite{Mo-02}).

We assume that $L=D\cap TS$ is a line bundle over $S$ (otherwise, in the general case, $L$ may have generic singularities): hence $S$ is either a torus or a Klein bottle.
Let $dP$ be the Popp volume on $M\setminus S$. Near $S$, we have $dP\sim d\nu\otimes\vert\frac{d\phi}{\phi}\vert$, where $\nu$ is a smooth measure on $S$ and $\phi=0$ is a local equation of $S$.

Given any smooth measure $\mu$ on $M$, we consider the corresponding sR Laplacian, and some eigenbasis of it.
We are able to obtain the following local Weyl law (see \cite{CHT-II}).

\begin{proposition}\label{prop_Martinet}
For every continuous function $f$ on $M$, we have
$$
\sum_{\lambda_n\leq\lambda}\int_M f\vert\phi_n\vert^2\,d\mu\sim\frac{\int_S f\,d\nu}{32} \lambda^2\log\lambda,
$$
as $\lambda\rightarrow+\infty$.
In particular, $N(\lambda)\sim \frac{\nu(S)}{32} \lambda^2\log\lambda$.
\end{proposition}

\begin{corollary}\label{cor_Martinet}
There exists a density-one subsequence $(n_j)_{j\in\N^*}$ such that
$$
\lim_{j\rightarrow+\infty} \int_M f\vert\phi_{n_j}\vert^2\,d\mu=0,
$$
for every continuous function $f$ on $M$, vanishing near $S$.
\end{corollary}

Note that Proposition \ref{prop_Martinet} cannot be obtained from the asymptotic estimates of the heat kernel along the diagonal, because those estimates are not uniform: along the diagonal, the asymptotics is in $t^{-5/2}$ on $S$, and in $C(q)t^{-2}$ outside of $S$, with $C(q)$ bounded but not integrable near $S$.
The proof of Proposition \ref{prop_Martinet} consists of first computing the asymptotic behavior of $\mathrm{Tr}(fe^{-t\triangle})$ as $t\rightarrow 0^+$ in the flat Martinet case, by rescaling. Then, using the locality of the heat kernel expansion along the diagonal, the result is obtained in the locally flat case. Using Dirichlet-Dirichlet bracketing, the result is then established in small cubes with a flat metric, and is finally extended using local quasi-isometries and minimax considerations.

It might be expected that the singular flow plays an important role in the asymptotic spectral properties of the sR Laplacian. In particular, it is tempting to think that a good assumption is the ergodicity of the abnormal geodesics (lifts of the singular curves in the cotangent space) in the characteristic distribution $\ker\omega_{\vert\Sigma}$, at least, under the assumption that there indeed exists a characteristic vector field leaving $\nu$ invariant.

\paragraph{Relationship between singular curves and magnetic fields.}
This paragraph is inspired by the paper \cite{Mo-95}. Let us consider a manifold $N$ of dimension $2$, endowed with a flat Riemannian metric and with its canonical Riemannian volume. We consider the Laplacian
$$
\triangle = - ( \partial_x -a \partial_z)^2 - (\partial_y -b \partial_z )^2  ,
$$
with $z \in \R/2\pi\Z $ and $A=a\,dx+b\,dy $ the magnetic potential on $N$. Setting $X=\partial_x -a \partial_z$ and $Y=\partial_y -b \partial_z $, we have $[X,Y]=B \,\partial_z $
where $dA=B\,dx\wedge dy$ is the magnetic field.
We assume that $B$ vanishes on a closed curve $\Gamma $ with a non-zero differential.
Then $\alpha =a\,dx +b\,dy + dz$ defines a Martinet distribution on $M = N\times \R/2\pi\Z $,  and the Martinet surface is $S=\Gamma \times \R/2\pi\Z$.
We have, then, $\alpha_g =\frac{1}{B  }(a\,dx +b\,dy + dz )$ and $\nu = \frac{1}{\Vert dB\Vert} \vert ds \, dz\vert $, where $s$ is the arc-length along $\Gamma$.
The previous expressions are still valid if the metric is not flat.
The characteristic direction is given by $A(\dot{\gamma}(s))\,ds +dz =0$ if $\Gamma$ is parametrized by the arc-length. It is always possible to choose a gauge $A$ near $\Gamma$ such that $A(\dot{\gamma})=A_0$ is constant. Using this gauge, the measure $\nu$ is invariant under the vector field $\Vert dB \Vert (\partial_s - A_0 \partial_z)$. The corresponding dynamics is ergodic if $2\pi A_0/\mathrm{length} (\Gamma)$ is irrational.

Another case of interest is the quasi-contact case in dimension $4$.
It is related to magnetic fields in dimension $3$. In this case, there exist some nontrivial singular curves that correspond to lines of a magnetic field.
It seems that the singularities can only occur at isolated points. We might then expect to have QE if there is only one such point, and otherwise a tunnel effect might occur between two different points.

\paragraph{Almost-Riemannian geometry.}
For the sake of simplicity, we restrict ourselves to almost-Riemannian (aR) structures on surfaces.
Let $M$ be a smooth compact connected manifold of dimension $2$. An aR structure on $M$ is locally given by two vector fields $X$ and $Y$ that generate, outside of a singular curve $S$, the tangent space of the surface $M$. We assume that, on $S$, the vector fields $X$, $Y$ and $[X,Y]$ generate $TM$.
The simplest example is the so-called \emph{Grushin case} and is given locally by $X=\partial_x$ and $Y=x\partial_y$. 
Outside of $S$, the vector fields $X$ and $Y$ define the Riemannian metric $g =\nu_X^2 + \nu_Y^2$, where $(\nu_X,\nu_Y)$ is dual to $(X,Y)$. The volume $\mu_{aR}$ associated with $g$ is singular along $S$.
In the Grushin case, we have $g=dx^2+(dy/x)^2$ and $d\mu_{aR}=\vert dx\, dy\vert/\vert x\vert$.

Let $d\mu$ be any volume form on $M$.
We consider the aR Laplacian $\triangle_{aR} = -X^\star X -Y^\star Y$, where the adjoints are taken with respect to $\mu$.
It follows from Remark \ref{rk:changemu1} that the aR Laplacians associated with two different volumes are unitary equivalent up to a potential of the order of $d(\cdot,S)^{-2}$. If $\mu=\mu_{aR}$ then the resulting operator is singular on $S$, however it is essentially selfadjoint on $M\setminus S$ (see \cite{BL-13}).

In this Grushin case, we prove in \cite{CHT-II} a result that is similar to the Martinet case, namely, that for every continuous function $f$ on $M$, we have
$$
\sum_{\lambda_n\leq\lambda}\int_M f\vert\phi_n\vert^2\,d\mu\sim\frac{\int_S f\,d\nu}{4\pi} \lambda\log\lambda,
$$
as $\lambda\rightarrow+\infty$ (with, as well, Corollary \ref{cor_Martinet} as a consequence).

\begin{remark}
In any of the above singular cases (Martinet, almost-Riemannian), one may think of applying the desingularization procedure of \cite{RS1976}. For instance, the Martinet case can be seen as the projection of the so-called Engel case, which refers to an equiregular horizontal distribution of rank $2$ in dimension $4$; the Grushin case can be seen as the projection of the 3D contact case. However, in view of obtaining a microlocal Weyl law, it is not clear whether or not this desingularization can be used in a relevant way, because when projecting one has to use integrals of the Schwartz kernel outside of the diagonal, that are not known.
\end{remark}

\begin{remark}
In order to study sR structures in which the rank of the horizontal distribution $D$ may not be constant, a more general definition of a sR structure has been introduced (although not exactly with those terms) in \cite{ABCGS2010}.

A relative tangent bundle on a manifold $M$ is a couple $(H,\xi)$, where $\pi:H\rightarrow M$ is a smooth fibration and $\xi:H\rightarrow TM$ is a smooth bundle morphism over $M$. The horizontal distribution is defined by $D=\xi(H)\subset TM$.
A vector field $X$ on $M$ is said to be horizontal if there exists a section $u$ of $H$ such that $X=\xi(u)$. An horizontal curve $q(\cdot):[0,1]\rightarrow M$ is an absolutely continuous curve for which there exists a section $u$ of $H$ such that $\dot q(t) = \xi_{q(t)}(u(t))$, for almost every $t\in[0,1]$.

A sR structure (with possibly varying rank) is then more generally defined by a quadruple $(M,H,\xi,g)$, where $M$ is a manifold, $(H,\xi)$ is a relative tangent bundle on $M$, and $g$ is a Riemannian metric on $H$.
This definition may be used in order to provide a unifying context for usual sR and aR geometries.
\end{remark}

\paragraph{Microlocal properties of SR wave equations.}
It is interesting to investigate properties of wave equations associated with a sR Laplacian.
One of the first questions that are arising is the question of microlocalization. More precisely, let us assume that the wave front of Cauchy data $(u_0,u_1)$ is a subset of a cone $K$ with compact base contained in an open cone $C$. Does there exist $\varepsilon(K,C) >0$ (not depending on the Cauchy data) such that the wave front of the solution is contained in $C$ for $\vert t\vert\leq \varepsilon$?
This question is a priori far from trivial because the sR cut-locus and the sR conjugate locus of a point always contain this point in their closure.
Another natural (simpler) question is to establish finite propagation of waves in the sR context.

\paragraph{Quantum limits.}
The decomposition and the statement on $\beta_0$ given in Theorem \ref{thm2} are actually valid for any sR Laplacian, not only in the 3D contact case. A stimulating open question is to establish invariance properties of the measure $\beta_\infty$ for more general sR geometries.

\end{document}